\documentclass[11pt,leqno]{amsart}
\usepackage{amssymb,amsmath,amsthm,amsfonts}
\usepackage{graphicx}

 \DeclareMathOperator{\diver}{div}

\newcommand{\field}[1]{\mathbb{#1}}
\newcommand{\G}{\field{G}}                      % Carnot group
\newcommand{\R}{\field{R}}                      % Reals
\newcommand{\Sph}{\field{S}}                    % Circle and spheres
\newcommand{\Heis}{\field{H}}                    % more Heis group

\newcommand{\fg}{{\mathfrak g}}
\newcommand{\cH}{{\mathcal H}}
\newcommand{\fh}{{\mathfrak h}}

\newcommand{\cL}{{\mathcal L}}

\newcommand{\fv}{{\mathfrak v}}

\newcommand{\ta}{{\tt a}}
\newcommand{\te}{{\tt e}}
\newcommand{\tg}{{\tt g}}
\renewcommand{\th}{{\tt h}}
\newcommand{\tx}{{\tt x}}

\newcommand{\boldN}{{\mathbf N}}

\newcommand{\deriv}[1]{{\frac{\partial}{\partial #1}}}

\newcommand{\bi}{{\mathbf i}}

\def\Barint_#1{\mathchoice
          {\mathop{\vrule width 6pt height 3 pt depth -2.5pt
                  \kern -8pt \intop}\nolimits_{#1}}%
          {\mathop{\vrule width 5pt height 3 pt depth -2.6pt
                  \kern -6pt \intop}\nolimits_{#1}}%
          {\mathop{\vrule width 5pt height 3 pt depth -2.6pt
                  \kern -6pt \intop}\nolimits_{#1}}%
          {\mathop{\vrule width 5pt height 3 pt depth -2.6pt
                  \kern -6pt \intop}\nolimits_{#1}}}

\theoremstyle{plain}
\newtheorem{theorem}{Theorem}
\newtheorem{corollary}[theorem]{Corollary}
\newtheorem{lemma}[theorem]{Lemma}

\theoremstyle{definition}
\newtheorem{definition}[theorem]{Definition}
\newtheorem{example}[theorem]{Example}
\newtheorem{remark}[theorem]{Remark}

\numberwithin{theorem}{section} \numberwithin{equation}{section}

\title{Polar coordinates in Carnot groups II}
\date{\today}
\author{Jeremy T. Tyson}
\address{Department of Mathematics \\ University of Illinois at Urbana-Champaign \\ 1409 West Green Street \\ Urbana, IL 61801}
\email{tyson@illinois.edu}

\begin{document}
\maketitle

\begin{abstract}
A Carnot group is polarizable if it carries a homogeneous norm whose powers are fundamental solutions for the $p$-sub-Lapacian operators for all $1<p\le \infty$. Such groups also support a system of horizontal polar coordinates. We prove that the converse statement is true: if a Carnot group supports a horizontal polar coordinate system with suitable properties, then it is polarizable.
\end{abstract}

\section{Introduction}

Carnot groups are nilpotent stratified Lie groups equipped with a sub-Riemannian geometric and analytic structure arising from the first (lowest) layer of the stratification of the Lie algebra. Carnot groups arise as local models for sub-Riemannian manifolds, and also serve as canonical non-smooth examples in the theory of analysis in metric measure spaces.

In this paper, we focus on a specific class of Carnot groups whose structure allows for explicit computations of sharp constants associated to first-order nonlinear potential theory, quasiconformal mappings, and Sobolev inequalities. These Carnot groups, introduced in \cite{bt:polar}, were named `polarizable' groups in view of a polar coordinate integration formula involving horizontal radial curves. This polar coordinate formula can be used, for instance, to compute explicit values for the moduli of spherical ring domains, which in turn yields sharp H\"older continuity exponents for quasiconformal maps. At the same time, these groups admit an explicit one-parameter family of fundamental solutions for the $p$-Laplace operators. In this note, we explore different equivalent definitions for the class of polarizable groups.

We denote by $\G$ a Carnot group of step $\iota$, with stratified Lie algebra $\fg = \fv_1 \oplus  \cdots \oplus \fv_\iota$. See section \ref{sec:background} for definitions and notation. The sub-Riemannian structure arises from a fixed inner product defined in the first (horizontal) layer $\fv_1$ of the Lie algebra. We fix an orthonormal basis $X_1,\ldots,X_m$ for $\fv_1$ and denote by $\nabla_0 u = (X_1 u, \ldots, X_m u)$ the horizontal gradient of a function $u$. Denoting by $\diver_0 V = X_1(a_1) + \cdots + X_m(a_m)$ the horizontal divergence of a horizontal vector field $V = a_1 \, X_1 + \cdots + a_m \, X_m$, we consider the {\it horizontal Laplacian}
\begin{equation}\label{eq:horiz-Laplacian}
\cL u = \diver_0(\nabla_0 u) = \sum_{i=1}^m X_i^2 u
\end{equation}
and the {\it horizontal $\infty$-Laplacian}
\begin{equation}\label{eq:horiz-infty-Laplacian}
\cL_\infty u = \tfrac12 \langle \nabla_0 ( || \nabla_0 u ||_0^2 ) , \nabla_0 u \rangle_0.
\end{equation}
Here $\langle\cdot,\cdot\rangle_0$ denotes the inner product in the horizontal tangent spaces, while $||\cdot||_0$ denotes the corresponding norm.

Carnot groups of step one are Euclidean spaces. In this case the horizontal distribution coincides with the full tangent bundle. The horizontal Laplacian and horizontal $\infty$-Laplacian are just the standard Laplacian $\triangle u = \sum_j \partial_j \partial_j u$ and $\infty$-Laplacian $\triangle_\infty u = \sum_{j,k} \partial_j \partial_k u \, \partial_j u \, \partial_k u$. It is well known that the fundamental solution for the Laplacian in $\R^n$, $n \ge 3$, is a constant multiple of $|x|^{2-n}$, where $|\cdot|$ denotes the Euclidean norm. At the same time, $\triangle_\infty(|\cdot|) = 0$ in the complement of $0 \in \R^n$. The $p$-Laplace operators $\triangle_p$, $1<p<\infty$, interpolate between the Laplacian and $\infty$-Laplacian, and extrapolate to the range $1<p<2$. For each $1<p<\infty$, the operator $\triangle_p$ admits a fundamental solution which is proportional to a suitable power of $|\cdot|$ (when $p=n$ this is replaced by $\log(1/|\cdot|)$).

The preceding story extends to other Carnot groups. By results of Folland \cite{fol:explicit},  Heinonen--Holopainen \cite{hh:carnot}, and Capogna--Danielli--Garofalo \cite{cdg:carnot}, fundamental solutions for the horizontal $p$-Laplacians on the Heisenberg group are given by suitable powers (or the logarithm) of the Kor\'anyi norm, and the Kor\'anyi norm itself is a solution to the horizontal $\infty$-Laplace equation in the complement of the origin. Moreover, the same story persists for Kaplan's H-type groups, a more general class of Carnot groups which retains many nice algebraic and analytic properties of the Heisenberg groups.

%Associated to each homogeneous norm $|\cdot|$ on $\G$, there is a polar coordinate integration formula for locally integrable functions on $\G$. This formula expresses the $L^1$ norm of a function $f$ on $%\G$ in terms of an integral of $f(\delta_\lambda(\ta))$, where $0<\lambda<\infty$ and $\ta$ lies in the unit sphere for the norm, i.e., $|\ta|=1$. Here $(\delta_\lambda)_{\lambda>0}$ denotes the usual %dilation semigroup on $\G$. While such polar coordinate formulas are useful for computing explicit integrals of radial functions, it has an important drawback: the dilation curves $\lambda \mapsto %\delta_\lambda(\ta)$ are typically not horizontal (and hence are not rectifiable).

As was first observed by Kor\'anyi and Reimann \cite{kr:rings}, the Heisenberg group $\Heis^n$ carries a system of {\it horizontal polar coordinates}. Kor\'anyi and Reimann used this system to compute the conformal capacities of ring domains, with concomitant implications for the regularity of quasiconformal maps. The paper \cite{bt:polar} introduced the class of polarizable Carnot groups as the class of Carnot groups in which the appropriate power of Folland's fundamental solution $u$ is annihilated by $\cL_{\infty}$ in the complement of the origin. Such groups admit an explicit one-parameter family of fundamental solutions for all of the horizontal $p$-Laplacians, and also a horizontal polar coordinate integration formula. Applications to explicit values for conformal capacities of spherical ring domains and regularity of quasiconformal mappings were given in \cite{bht:singular} and \cite{bmv:modules}. Other uses of horizontal polar coordinates in Carnot groups include sharp constants in Hardy-type inequalities \cite{gv:hardy}, extremality for quasiconformal mappings and mappings with controlled dilatation \cite{bfp:teichmuller}, boundary behavior of quasiconformal mappings \cite{aw:prime-ends}, and Hardy spaces in the Heisenberg group \cite{af:hardy}. Recently, Koskela and Nguyen \cite{kn:polar} introduced an abstract notion of rectifiable polar coordinate decomposition in metric measure spaces, with applications to the asymptotic behavior of homogeneous Sobolev functions. 

The main theorem of this note resolves an axiomatic question left unaddressed in \cite{bt:polar}. We denote by $u$ the fundamental solution for the Laplacian $\cL$ defined in \eqref{eq:horiz-Laplacian}. By a result of Folland \cite{fol:subelliptic}, such a fundamental solution exists in any Carnot group, moreover, $u$ is positive and finite in $\G \setminus \{ 0\}$ and $u$ is homogeneous of order $2-Q$ with respect to the dilation semigroup, where $Q$ denotes the homogeneous dimension of $\G$. Consequently, $\boldN_u := u^{1/(2-Q)}$ is a homogeneous quasi-norm on $\G$, i.e., it is homogeneous of order one. We say that $\G$ {\it admits a coherent family of singular solutions to the horizontal $p$-Laplacians for all $1<p\le\infty$} if a suitable power of $\boldN_u$ is a fundamental solution for the $p$-Laplacian
\begin{equation}\label{eq:horiz-p-Laplacian}
\cL_p f = \diver_0(||\nabla_0 f||_0^{p-2} \, \nabla_0 f)
\end{equation}
for each $p \in (1,\infty]$ (when $p=Q$ we use instead the logarithm of $\boldN_u$). We say that $\G$ {\it admits horizontal polar coordinates with respect to $\boldN_u$} if an open subset of full measure in $\G$ is foliated by a family of horizontal curves 
$$
\gamma_\ta:(0,\infty) \to \G, \qquad \ta \in S := \{ \boldN = 1 \}
$$
satisfying certain additional properties, and there exists a Radon measure $\sigma$ on $S$ so that the integration formula
$$
\int_\G f(g) \, dg = \int_S \int_0^\infty f(\gamma_\ta(s)) \, s^{Q-1} \, ds \, d\sigma(\ta)
$$
holds for all $f \in L^1(\G)$. Precise definitions for these two concepts appear in subsection \ref{subsec:polarizable}. With this terminology in place we state the following theorem.

\begin{theorem}\label{th:main}
Let $\G$ be a Carnot group of homogeneous dimension $Q>2$. Let $\boldN_u$ denote Folland's homogeneous quasi-norm. Then the following are equivalent.
\begin{itemize}
\item[(i)] $\cL_\infty(\boldN_u) = 0$ in the complement of the origin $0$,
\item[(ii)] $\G$ admits a coherent family of singular solutions to the horizontal $p$-Laplacians for all $1<p\le\infty$,
\item[(iii)] $\G$ admits horizontal polar coordinates with respect to $\boldN_u$.
\end{itemize}
\end{theorem}

Consequently, a Carnot group $\G$ is polarizable if and only if it satisfies any one of the equivalent conditions in Theorem \ref{th:main}. In \cite{bt:polar}, condition (i) was used as the definition of polarizability.

The implication (ii) $\Rightarrow$ (i) is trivial, and the implications (i) $\Rightarrow$ (ii) $\Rightarrow$ (iii) were proved in \cite{bt:polar}. It suffices for us to prove that (iii) $\Rightarrow$ (i). We will do this in section \ref{sec:proof1}.

It is natural to ask how extensive is the class of polarizable groups. In particular, {\it are there polarizable groups which are not of H-type?} We conjecture that the answer is no. In a companion paper \cite{tys:h-type-stability} we propose a new approach towards this conjecture. We introduce a numerical measurement, denoted $\delta_0(\G)$, for a step two Carnot group $\G$ which quantifies the extent to which the group deviates from the class of H-type groups. That is, $\delta_0(\G) = 0$ if and only if $\G$ is of H-type.

A result of Beals, Gaveau and Greiner provides an explicit integral representation for Folland's fundamental solution in any step two group. We can therefore express the horizontal $\infty$-Laplacian of the corresponding homogeneous norm explicitly as an iterated integral involving first and second horizontal derivatives of the integrand in the Beals--Gaveau--Greiner formula. The group is polarizable if the resulting expression is identically equal to zero, that is, if
\begin{equation}\label{eq:infty-Lap-equation}
(\cL_{\infty}\boldN_u)(\tg) = 0 \qquad \forall \, \tg \in \G, \tg \ne \te.
\end{equation}
One potential approach to the above conjecture, in the setting of groups of step two, would be to control $\delta_0(\G)$ from above by a suitable expression involving the values of $(\cL_{\infty}\boldN_u)(\tg)$ for certain points $\tg \in \G$, so that the upper bound vanishes if \eqref{eq:infty-Lap-equation} is satisfied. We investigate this idea in more detail in \cite{tys:h-type-stability}, where we formulate a conjecture along these lines and verify this conjecture for a specific infinite family of step two Carnot groups. In \cite{tys:h-type-stability} we also use a similar method  to establish several other new analytic characterizations for the class of H-type groups.

\smallskip

\paragraph{\bf Acknowledgements.} The author acknowledges support from the Simons Foundation under grant \#852888, `Geometric mapping theory and geometric measure theory in sub-Riemannian and metric spaces'. In addition, this material is based upon work supported by and while the author was serving as a Program Director at the National Science Foundation. Any opinion, findings, and conclusions or recommendations expressed in this material are those of the author and do not necessarily reflect the views of the National Science Foundation.

\section{Background, notation and definitions}\label{sec:background}

Let $\G$ be a Carnot group with Lie algebra
$$
\fg = \fv_1 \oplus \cdots \oplus \fv_\iota \simeq \R^N
$$
and sub-Riemannian metric $\langle \cdot,\cdot \rangle_0$ on $\fv_1$ associated to a fixed orthonormal basis $X_1,\ldots,X_m$ for $\fv_1$. Denote by $d_{cc}$ the Carnot--Carath\'eodory metric induced by $\langle \cdot,\cdot \rangle_0$, and let $|\tg|_{cc} := d_{cc}(\tg,\te)$, $\tg \in \G$, where $\te$ denotes the identity element of $\G$.

Denote by
$$
\cL= X_1^2 + \cdots + X_m^2
$$
the sub-Laplacian. By a result of Folland \cite{fol:subelliptic}, $\cL$ admits a fundamental solution $u$. By H\"ormander's hypoellipticity criterion, $u$ is $C^\infty$ in $\G \setminus \{\te\}$. Furthermore, $\lim_{\tg \to \te} u(\tg) = +\infty$, $\lim_{|\tg|_{cc} \to \infty} u(\tg) = 0$, and $u$ is homogeneous of order $2-Q$ with respect to the standard dilations $(\delta_\lambda)$.

Recall that $\delta_\lambda:\G\to\G$ is the dilation with scale factor $\lambda>0$, defined on the level of the Lie algebra by $\delta_\lambda|_{\fv_j} = \lambda^j$. The Jacobian determinant of $\delta_\lambda$ is constant and equal to $\lambda^Q$, where
$$
Q = \sum_{j=1}^\iota j \, \dim \fv_j
$$
is the {\it homogeneous dimension} of $\G$.

Let $\mu$ be Haar measure in $\G$; it coincides (up to a multiplicative constant) with Lebesgue measure $\cL^N$ in $\G\simeq\R^N$ or with the Carnot--Carath\'eodory $Q$-dimensional Hausdorff measure $\cH^Q_{cc}$.

The order $(2-Q)$-homogeneity of Folland's fundamental solution $u$ means that
$$
u \circ \delta_\lambda = \lambda^{2-Q} u
$$
for all $\lambda>0$. Define
$$
\boldN_u:= u^{1/(2-Q)};
$$
this is {\it Folland's homogeneous quasi-norm} on $\G$. It is homogeneous of order one, i.e., $\boldN_u \circ \delta_\lambda = \lambda \boldN_u$ for all $\lambda>0$. It may not be a true norm on $\G$, but it is at least a quasinorm, i.e., there exists a constant $C>0$ such that $\boldN_u(\tg*\th) \le C(\boldN_u(\tg)+\boldN_u(\th))$ for all $\tg,\th \in \G$.

\begin{example}[Folland, \cite{fol:explicit}]\label{ex:Folland-Heisenberg}
Denote by $\Heis^n$ the $n$th Heisenberg group. The underlying space is $\R^{2n+1}$, with coordinates $(x,t) = (x_1,\ldots,x_{2n},t)$ and group law $(x,t)*(x',t') = (x+x',t+t'+2\sum_{j=1}^n (x_jx_{n+j}'-x_{n+j}x_j')$. The Lie algebra $\fh_n$ is stratified as $\fv_1\oplus\fv_2$, where $\dim\fv_1 = 2n$ and $\dim\fv_2 = 1$. A basis of left-invariant vector fields for $\fv_1$ is given by
$$
X_j = \deriv{x_j} + 2x_{n+j} \deriv{t} \qquad \mbox{and} \qquad X_{n+j} = \deriv{x_{n+j}} - 2x_j \deriv{t}, \quad j=1,\ldots,n
$$
and we record the commutation relations
$$
[X_j,X_{n+j}] = -4T, \qquad j=1,\ldots,n ,
$$
where $T = \deriv{t}$. For a suitable choice of $c(n)>0$, the function
$$
u(x,t) = \frac{c(n)}{(||x||^4 + t^2)^{n/2}}
$$
is a fundamental solution for $-\cL$, where $\cL$ denotes the sub-Laplacian $\sum_{j=1}^{2n} X_j^2$. The corresponding norm is a multiple of the {\it Kor\'anyi norm}
$$
\boldN(x,t) = (||x||^4 + t^2)^{1/4}.
$$
In this example, $\boldN$ is actually a norm, not just a quasinorm.
\end{example}

\subsection{Nonlinear potential theory in Carnot groups}

Let $\G$ be an arbitrary Carnot group. For $1<p<\infty$ the {\it $p$-Laplacian} in $\G$ is defined to be
$$
\cL_p f := \diver_0 ( ||\nabla_0 f||^{p-2} \, \nabla_0 f )
$$
where $\nabla_0 f = (X_1f,\ldots,X_mf)$ and $\diver_0(\sum_{j=1}^m a_j X_j) = X_j(a_j)$ denote the horizontal gradient of a real-valued function $f$ and the horizontal divergence of a horizontal vector field, respectively. When $p=\infty$ we use instead the {\it $\infty$-Laplacian} $\cL_\infty f$, which arises as a suitable limit of renormalized $p$-Laplacians as $p\to\infty$. See \eqref{eq:horiz-infty-Laplacian} for the definition of $\cL_\infty$.

In $\Heis^n$ it is known by work of Capogna, Danielli and Garofalo \cite{cdg:carnot} (see also Heinonen and Holopainen \cite{hh:carnot} for the case $p=Q$) that for each $1<p<\infty$, there exists a constant $c(n,p)>0$ so that the function
$$
u_p(x,t) = \begin{cases}
c(n,p) \, \boldN(x,t)^{(p-Q)/(p-1)}, &p\ne Q, \\
c(n,Q) \, \log(1/\boldN(x,t)), &p=Q
\end{cases}
$$
is a fundamental solution for $-\cL_p$. Here $\boldN$ denotes the Kor\'anyi norm in $\Heis^n$ as defined in Example \ref{ex:Folland-Heisenberg}. Since $\cL_p$ is nonlinear, the notion of fundamental solution must be understood appropriately. The preceding assertion means that the representation formula
$$
\varphi(\te) = \int_{\Heis^n} ||\nabla_0 u_p||_0^{p-2} \, \langle \nabla_0 u_p , \nabla_0 \varphi \rangle_0 \, d\mu
$$
holds for all $\varphi \in C^\infty_0(\Heis^n)$. Furthermore, $\boldN$ itself is $\infty$-harmonic in the complement of the origin:
$$
\cL_\infty \boldN = 0 \qquad \mbox{in $\Heis^n \setminus \{\te\}$.}
$$
Kaplan \cite{kap:h-type} introduced a more general class of step two Carnot groups, known as {\it H-type groups} (or {\it Heisenberg-type groups}). All of the Heisenberg groups are H-type groups, but there exist H-type groups whose center has arbitrarily large dimension. All of the preceding discussion (including the existence of explicit fundamental solutions to the $p$-Laplacians) extends to all H-type groups.

\subsection{Polar coordinates in Carnot groups}

Let $\G$ be an arbitrary Carnot group equipped with a fixed homogeneous norm $|\cdot|$. Denote by $S := \{ \tg \in \G \, : \, |\tg| = 1 \}$ the unit sphere for this norm. It is known (see Proposition 1.15 in \cite{fs:hardy}) that there is a polar coordinate integration formula for locally integrable functions on $\G$ adapted to the norm $||\cdot||$ and the dilation semigroup $(\delta_\lambda)$. More precisely, there exists a Radon measure $\tilde\sigma$ on $S$ so that
\begin{equation}\label{eq:FS-polar-coordinates}
\int_{\G} f(\tg) \, d\mu(\tg) = \int_S \int_0^\infty f(\delta_\lambda(\ta)) \, \lambda^{Q-1} \, d\lambda \, d\tilde\sigma(\ta)
\end{equation}
for all $f \in L^1(\G)$. Equation \eqref{eq:FS-polar-coordinates} is useful for computing explicit integrals of radial functions on $\G$. However, this type of polar coordinate formula has an important drawback: the dilation curves $\lambda \mapsto \delta_\lambda(\ta)$, $\ta \in S$, are typically not horizontal (and hence rectifiable). This makes \eqref{eq:FS-polar-coordinates} useless for purposes such as computing the capacities of spherical ring domains, or identifying extremal functions for certain variational problems.

In the Heisenberg group $\Heis^n$, there exists a system of {\it horizontal polar coordinates}. Kor\'anyi and Reimann \cite{kr:rings} were the first to introduce this system; they used it to compute the conformal capacities of ring domains, with concomitant implications for the regularity theory of quasiconformal maps. Let $\boldN$ be the Kor\'anyi norm in $\Heis^n$ and denote by $S$ the unit $\boldN$-sphere. Let $Z = \{(0,t):t \in \R\}$ denote the vertical axis in $\Heis^n$; note that $Z$ is a $\mu$-null set. For $\ta = (z,t) \in S \setminus Z$ and $s>0$, define
$$
\gamma_\ta(s) = (sze^{-\bi\alpha\log(s)},s^2t), \qquad \mbox{where $\alpha = t/|z|^2$.}
$$
The curves $\gamma_\ta$, $\ta \in S \setminus Z$, are horizontal, constant speed curves; the union of the images of all of these curves is precisely $\Heis^n \setminus Z$. More generally, for $\tx = (z,t) \in \Heis^n \setminus Z$ and $s>0$ define
$$
\varphi(s,\tx) = (sze^{-\bi\alpha\log(s)},s^2t)
$$
where $\alpha$ is as before. Then $\varphi:(0,\infty) \times \Heis^n \setminus Z \to \Heis^n \setminus Z$ defines a flow whose flowlines are precisely the curves $\gamma_\ta$, $\ta \in S \setminus Z$. Furthermore, there exists a unique Radon measure $\sigma$ on $S$ so that
\begin{equation}\label{eq:KR-polar-coordinates}
\int_{\Heis^n} f(\tx) \, d\mu(\tx) = \int_S \int_0^\infty f(\varphi(s,\ta)) \, s^{2n+1} \, ds \, d\sigma(\ta) \quad \mbox{for all $f \in L^1(\G)$.}
\end{equation}

\section{Polarizable groups}\label{subsec:polarizable}

The notion of {\it polarizable Carnot group} was introduced in \cite{bt:polar} in an attempt to understand the precise context in which the aforementioned features of Heisenberg (and more generally, H-type) sub-Riemannian geometry and analysis reside.

\begin{definition}
A Carnot group $\G$ is {\it polarizable} if 
$$
\cL_\infty \boldN_u=0
$$
in the complement of $\te$ in $\G$. Here $\boldN_u$ denotes Folland's homogeneous quasi-norm on $\G$.
\end{definition}

It is clear that the Heisenberg groups (more generally, the H-type groups) are polarizable. One of the main results of \cite{bt:polar} was the following theorem.

\begin{theorem}[Polarizable Carnot groups admit horizontal polar coordinates]\label{th:pCg-hpc}
Let $\G$ be a polarizable Carnot group with associated homogeneous quasi-norm $\boldN_u$. Let $S := \{ \tg \in \G \, : \, \boldN_u(\tg) = 1 \}$ be the unit $\boldN_u$-sphere. Then there exists a $\mu$-null set $Z \subset \G$ and a foliation
$$
\G \setminus Z = \coprod_{\ta \in S \setminus Z} \gamma_\ta,
$$
where, for each $\ta \in S \setminus Z$, $\gamma_\ta:(0,\infty) \to \G$ is a horizontal curve with $\gamma_\ta(1) = \ta$ and $\lim_{s \to 0}\gamma_\ta(s) = \te$. Moreover, the following conclusions hold for all $\ta \in S \setminus Z$:
\begin{itemize}
\item[(i)] $\boldN_u(\gamma_\ta(s)) = s$ for all $s>0$,
\item[(ii)] $\dot\gamma_\ta(s)$ is a multiple of $\nabla_0\boldN_u(\gamma_\ta(s))$ for all $s>0$, and
\item[(iii)] the curve $\gamma_\ta$ has constant speed, i.e., $||\dot\gamma_\ta(s)||_0$ is independent of $s$.
\end{itemize}
Finally, there exists a unique Radon measure $\sigma$ on $S$ so that
$$
\int_\G f(\tg) \, d\tg = \int_S \int_0^\infty f(\gamma_\ta(s)) \, s^{Q-1} \, ds \, d\sigma(\ta)
$$
for all $f \in L^1(\G)$. 
\end{theorem}

\noindent Here and henceforth we abbreviate $d\tg = d\mu(\tg)$ for integration with respect to Haar measure in $\G$. For the purposes of later discussion, we observe that the proof of Theorem \ref{th:pCg-hpc} in \cite{bt:polar} yields the following formula for the speed of the polar curve $\gamma_\ta$:
\begin{equation}\label{eq:speed}
||\dot\gamma_\ta(s)||_0 = \frac1{||\nabla_0\boldN_u(\ta)||_0}
\end{equation}
for all $s>0$ and all $\ta \in S \setminus Z$.

\smallskip

To further clarify the relationship between these various concepts, we introduce the following definitions.

\begin{definition}[Singular solutions to $p$-Laplacians]\label{def:p-lap-sing}
A Carnot group $\G$ {\it admits a coherent family of singular solutions to the $p$-Laplacians} if, for each $1<p\le\infty$, the function
$$
u_p(\tg) = \begin{cases}
\boldN_u(\tg)^{(p-Q)/(p-1)}, &p\ne Q, \\
\log(1/\boldN_u(\tg)), &p=Q,
\end{cases}
$$
satisfies $\cL_p u_p=0$ in the complement of the origin $0 \in \G$. Here we interpret $u_\infty(\tg) = \boldN_u(\tg)$.
\end{definition}

\begin{definition}[Horizontal polar coordinates]\label{def:hpc}
Let $\G$ be a Carnot group equipped with Haar measure $\mu$ and a fixed quasi-norm $|\cdot|$, and let $S = \{ \tg \in \G : |\tg| = 1\}$ be the unit sphere. Assume that the function $|\cdot|:\G\to\R$ is horizontally $C^1$ in the complement of $0$. We say that $\G$ {\it admits horizontal polar coordinates with respect to $|\cdot|$} if there exists a $\mu$-null set $Z \subset \G$ and a foliation
$$
\G \setminus Z = \coprod_{\ta \in S \setminus Z} \gamma_\ta,
$$
where, for each $\ta \in S \setminus Z$, $\gamma_\ta:(0,\infty) \to \G$ is a horizontal curve with $\gamma_\ta(1) = \ta$. Moreover, the following conclusions hold for all $\ta \in S \setminus Z$:
\begin{itemize}
\item[(i)] $|\gamma_\ta(s)| = s$ for all $s>0$,
\item[(ii)] $\dot\gamma_\ta(s)$ is a multiple of $(\nabla_0|\cdot|)(\gamma_\ta(s))$ for all $s>0$, and
\item[(iii)] $||\dot\gamma_\ta(s)||_0$ is independent of $s$.
\end{itemize}
Furthermore, there exists a unique Radon measure $\sigma$ on $S$ so that
$$
\int_\G f(\tg) \, d\tg = \int_S \int_0^\infty f(\gamma_\ta(s)) \, s^{Q-1} \, ds \, d\sigma(\ta)
$$
for all $f \in L^1(\G)$.
\end{definition}

Here we say that a function $f:\G \to \R$ is {\bf horizontally $C^1$} if all first-order horizontal derivatives $X_j f$, $j=1,\ldots,m$, exist and are continuous.

\begin{remark}
We do not know whether all three conditions (i)--(iii) in Definition \ref{def:hpc} are necessary. In particular, it would be interesting to know whether condition (iii) can be derived as a consequence of the remaining assumptions in that definition.
\end{remark}

\section{Proof of Theorem \ref{th:main}}\label{sec:proof1}

As mentioned above, it suffices to prove the implication (iii) $\Rightarrow$ (i). Let $\G$ be a Carnot group of homogeneous dimension $Q>2$, let $\boldN_u$ denote Folland's homogeneous quasi-norm on $\G$. We assume that $\G$ admits horizontal polar coordinates with respect to $\boldN_u$, and let $Z$ be the associated exceptional set as indicated in Definition \ref{def:hpc}. We begin by defining a flow on $\G\setminus Z$ associated to the polar coordinate decomposition. We will then observe that the governing ODE for this flow, coupled with the polar coordinate integration formula, yields a divergence-form PDE satisfied by $\boldN_u$. Using this PDE together with the fact that $\boldN_u^{2-Q}$ is $2$-harmonic outside $0$, we deduce that $\boldN_u$ itself is $\infty$-harmonic outside $0$.

To this end, we define a function
$$
\varphi:(0,\infty) \times \G \setminus Z \longrightarrow \G \setminus Z
$$
as follows. Let $\tg \in \G \setminus Z$ and write $\tg = \gamma_\ta(t)$ for some unique choice of $\ta \in S \setminus Z$ and $t>0$. Note that $\boldN_u(\tg) = t$. For $s>0$ define
$$
\varphi(s,\tg) = \gamma_\ta(st).
$$
It is clear that $\varphi(1,\tg) = \tg$ for all $g \in \G \setminus Z$.

\begin{lemma}\label{lem:phi-props}
The map $\varphi$ satisfies the following properties:
\begin{itemize}
\item[(1)] $\varphi(s_1,\varphi(s_2,\tg)) = \varphi(s_1s_2,\tg)$ for all $\tg \in \G \setminus Z$ and $s_1,s_2>0$.
\item[(2)] $\boldN_u(\varphi(s,\tg)) = s \, \boldN_u(\tg)$ for all $\tg \in \G \setminus Z$ and $s>0$.
\item[(3)] The ODE
$$
\frac{\partial\varphi}{\partial s}(s,\tg) = \frac{\boldN_u(\varphi(s,\tg))}{s} \, \frac{\nabla_0\boldN_u(\varphi(s,\tg))}{||\nabla_0\boldN_u(\varphi(s,\tg))||_0^2}
$$
is satisfied in $(0,\infty) \times \G \setminus Z$.
\end{itemize}
\end{lemma}

\begin{proof}
(1) is clear from the definition. From (1) we conclude that $\varphi$ defines a flow on $\G \setminus Z$, whose flowlines are precisely the curves $\gamma_\ta$, $\ta \in S \setminus Z$. To see why (2) holds, we write $\tg = \gamma_\ta(t)$ and compute
$$
\boldN_u(\varphi(s,\tg)) = \boldN_u(\gamma_\ta(st)) = st = s \, \boldN_u(\tg).
$$
We turn to item (3). Again, we fix $\tg =\gamma_\ta(s) \in \G \setminus Z$. Differentiating $\varphi(s,\tg)$ with respect to $s$ gives
$$
\frac{\partial\varphi}{\partial s}(s,\tg) = t \, \dot\gamma_\ta(st)
$$
and we conclude from Axiom (ii) in Definition \ref{def:hpc} that $\tfrac{\partial\varphi}{\partial s}(s,\tg)$ is a multiple of $\nabla_0\boldN_u(\varphi(s,\tg))$. Differentiating the identity in (2) with respect to $s$ gives
$$
\langle \nabla_0\boldN_u(\varphi(s,\tg)),\frac{\partial\varphi}{\partial s}(s,\tg) \rangle_0 = \frac{\partial}{\partial s} \boldN_u(\varphi(s,\tg)) = \boldN_u(\tg) = \frac{\boldN_u(\varphi(s,\tg))}{s}.
$$
Since $\tfrac{\partial\varphi}{\partial s}(s,\tg)$ is a multiple of $\nabla_0\boldN_u(\varphi(s,\tg))$, we conclude that
$$
\frac{\partial\varphi}{\partial s}(s,\tg) = \frac{\boldN_u(\varphi(s,\tg))}{s} \, \frac{\nabla_0 \boldN_u(\varphi(s,\tg))}{||\nabla_0 \boldN_u(\varphi(s,\tg))||_0^2}
$$
as desired.
\end{proof}

\begin{lemma}\label{lem:vol}
The identity
$$
\det D_\th \varphi(t,\th) = t^Q
$$
holds for all $t>0$ and all $\th \in \G \setminus Z$.
\end{lemma}

\begin{proof}
We start from the polar coordinate integration formula
\begin{equation}\label{eq:pc}
\int_\G f(\tg) \, d\tg = \int_S \int_0^\infty f(\varphi(s,\ta)) \, s^{Q-1} \, ds \, d\sigma(\ta)
\end{equation}
valid for all integrable functions $f:\G\to\R$.

Fix $t>0$ and substitute $\tg = \varphi(t,\th)$ on the left hand side of \eqref{eq:pc} and $s = tu$ on the right hand side of \eqref{eq:pc}. We obtain
\begin{equation*}\begin{split}
\int_\G f(\varphi(t,\th)) \, \det D_\th \varphi(t,\th) \, d\th
&= \int_S \int_0^\infty f(\varphi(tu,\ta)) t^Q u^{Q-1} \, du \, d\sigma(\ta) \\
&= \int_S \int_0^\infty f(\varphi(t,\varphi(u,\ta))) u^{Q-1} \, du \, d\sigma(\ta) \, t^Q.
\end{split}\end{equation*}
Setting $\tilde{f}(\th) = f(\varphi(t,\th))$ we rewrite the preceding equation in the form
$$
\int_\G \tilde{f}(\th) \, \det D_h \varphi(t,\th) \, d\th = t^Q \int_S \int_0^\infty \tilde{f}(\varphi(u,\ta)) u^{Q-1} \, du \, d\sigma(\ta) = t^Q \int_\G \tilde{f}(\th) \, d\th.
$$
Since this holds for all integrable $\tilde{f}:\G\to\R$ we conclude that
$$
\det D_\th \varphi(t,\th) = t^Q
$$
for all $t>0$ and $\th \in \G \setminus Z$, as desired.
\end{proof}

As in \cite{bt:polar}, we now set
$$
F(\tx,s) = \frac{\boldN_u(\tx)}{s} \, \frac{\nabla_0 \boldN_u(\tx)}{||\nabla_0 \boldN_u(\tx)||_0^2}
$$
for $(s,\tx) \in (0,\infty) \times \R^N \setminus Z$ (where we identified $\G$ with $\R^N$). Lemma \ref{lem:phi-props}(3) tells us that
$$
\frac{\partial\varphi}{\partial s}(s,\tg) = F(\varphi(s,\tg),s).
$$
Differentiating in the spatial ($x$-)direction gives
$$
D_\tg \frac{\partial\varphi}{\partial s}(s,\tg) = D_\tx F(\varphi(s,\tg),s) \cdot D_\tg \varphi(s,\tg)
$$
and using the equality of mixed partial derivatives we deduce that
$$
\frac{\partial}{\partial s} D_\tg \varphi(s,\tg) = D_\tx F(\varphi(s,\tg),s) \cdot D_\tg \varphi(s,\tg).
$$
Taking determinants and using a standard determinantal identity we conclude that
$$
\frac{\partial}{\partial s} \det D_\tg \varphi(s,\tg) = \diver F(\varphi(s,\tg),s) \, \det D_\tg \varphi(s,\tg).
$$
By Proposition 7.1 in \cite{bt:polar}, the Euclidean divergence $\diver$ and the horizontal divergence $\diver_0$ agree on horizontal vector fields, so we can rewrite the above equation in the form
$$
\frac{\partial}{\partial s} \det D_\tg \varphi(s,\tg) = \diver_0 F(\varphi(s,\tg),s) \, \det D_\tg \varphi(s,\tg).
$$
By Lemma \ref{lem:vol}, the preceding equation simplifies to
$$
Q s^{Q-1} = \diver_0 F(\varphi(s,\tg),s) \, s^Q
$$
or
$$
\diver_0 F(\varphi(s,\tg),s) = \frac{Q}{s}.
$$
Setting $\tx = \varphi(s,\tg)$ we conclude that
$$
\diver_0 F(\tx,s) = \frac{Q}{s}
$$
for all $s>0$ and all $\tx \in \G \setminus Z$. Recalling the definition of the vector field $F$, we conclude that
\begin{equation}\label{eq:1}
\diver_0 \left( \frac{\boldN_u}{||\nabla_0\boldN_u||_0^2} \, \nabla_0 \boldN_u \right) = Q
\end{equation}
in $\G\setminus Z$.

Since $u = \boldN_u^{2-Q}$ and $\cL(u) = 0$ in the complement of $\te$, we deduce that
\begin{equation}\label{eq:2}
\cL(\boldN_u) = (Q-1) \frac{||\nabla_0 \boldN_u||_0^2}{\boldN_u}
\end{equation}
in the complement of $\te$. Combining \eqref{eq:1} and \eqref{eq:2} we conclude that
\begin{equation*}\begin{split}
Q &= \diver_0 \left( \frac{\boldN_u}{||\nabla_0 \boldN_u||_0^2} \, \nabla_0 \boldN_u \right) \\
&= \frac{\boldN_u}{||\nabla_0 \boldN_u||_0^2} \, \cL(\boldN_u) + 1 - \frac{\boldN_u \langle \nabla_0 ||\nabla_0 \boldN_u||_0^2 , \nabla_0 \boldN_u \rangle_0}{||\nabla_0 \boldN_u||_0^4} \\
&= (Q-1) + 1 - \frac{2\boldN_u \, \cL_\infty(\boldN_u)}{||\nabla_0\boldN_u||_0^4}
\end{split}\end{equation*}
from which it follows that $\cL_\infty \boldN_u = 0$. Note that this conclusion holds {\it a priori} in $\G \setminus Z$ (since that is where \eqref{eq:1} is valid), however, since $\boldN_u$ is $C^\infty$ in the complement of $0$, the desired conclusion extends by continuity to all of $\G \setminus \{\te\}$. This completes the proof of Theorem \ref{th:main}.

\section{Concluding remarks}

In their study \cite{kn:polar} of the asymptotic behavior of homogeneous Sobolev functions, Koskela and Nguyen introduce the notion of a (rectifiable) polar coordinate system on a metric measure space. The horizontal polar coordinates in a polarizable Carnot group provide one natural example of such a system. In this section, we clarify a small but important technical point regarding the validity of the axioms in \cite{kn:polar} for the horizontal polar coordinate system on a polarizable group. To wit, in \cite{kn:polar} it is assumed that all rectifiable curves are parameterized by arc length. However, in a polarizable Carnot group the polar curves are canonically parametrized at constant speed, but not necessarily unit speed. In order to bring these two descriptions into alignment, we reformulate the polar coordinate system in a polarizable group using arc length parametrized curves, and indicate how this yields a (strong, rectifiable) polar coordinate system in the sense of \cite{kn:polar}. We then restate the conclusions from \cite{kn:polar} which derive from the existence of such polar coordinates.

\subsection{Rectifiable polar coordinates in metric measure spaces}

We briefly review the main results of \cite{kn:polar}. The setting is a metric measure space $(X,d,\mu)$, where $\mu$ is a Borel regular measure on a metric space $(X,d)$. Following Heinonen and Koskela \cite{hk:quasi}, a Borel function $g:X \to [0,\infty]$ is said to be an {\it upper gradient} of a function $u:X \to \R$ if the inequality $|u(x)-u(y)| \le \int_\gamma g \, ds$ holds for each rectifiable curve $\gamma$ in $X$ with endpoints $x$ and $y$. The {\it homogeneous Sobolev space} $\dot{N}^{1,p}(X)$ consists of all locally integrable functions $u$ on $X$ that have an upper gradient $g$ in $L^p(X)$. 

A curve $\gamma:[0,\infty) \to X$ is called an {\it infinite curve} if $\gamma$ is not contained in any metric ball of $X$. The collection of all infinite curves in $X$ is denoted $\Gamma^\infty(X)$. The results of \cite{kn:polar} give conditions which imply the existence of the asymptotic limit $\lim_{s\to\infty} u(\gamma(s))$ for infinite curves $\gamma$ and homogeneous Sobolev functions $u \in \dot{N}^{1,p}(X)$. Several of these conditions are stated in terms of the $p$-modulus of $\Gamma^\infty$ and certain distinguished subfamilies of $\Gamma^\infty$.

To simplify the exposition, we restrict attention to Ahlfors $Q$-regular metric measure spaces $(X,d,\mu)$ which support the $p$-Poincar\'e inequality for some $1\le p<Q$. Note that each Carnot group, equipped with its Carnot--Carath\'eodory metric and Haar measure, satisfies these assumptions. We refer the reader to \cite{hk:quasi} or \cite[Chapter 14]{hkst:book} for further details. We emphasize that the results of \cite{kn:polar} are formulated in a more general setting. 

Let $\Sph$ be an abstract parametrizing set equipped with a Radon measure $\sigma$, and let $o \in X$ denote a fixed basepoint. Let $\Gamma^\infty_o(\Sph)$ denote a collection of rectifiable curves $\gamma_\xi \in \Gamma^\infty$, indexed by $\xi \in \Sph$, with $\gamma_\xi(0) = o$. As previously mentioned, all rectifiable curves are assumed to be parametrized by arc length. The data $(\Sph,\sigma,\Gamma^\infty_o(\Sph))$ is said to define a {\it strong (rectifiable) polar coordinate system} on $(X,d,\mu)$ if the identity
\begin{equation}\label{eq:rpc}
\int_X u(x) \, d\mu(x) = \int_{\Sph} \int_0^\infty u(\gamma_\xi(s)) \, s^{Q-1} \, ds \, d\sigma(\xi) 
\end{equation}
is satisfied for all $u \in L^1(X)$.

\subsection{Rectifiable vs.\ horizontal polar coordinates in polarizable Carnot groups}

Now assume that $\G$ is a polarizable Carnot group equipped with Carnot--Carath\'eodory metric $d_{cc}$ and Haar measure $\mu$. Then $(\G,d_{cc},\mu)$ is a metric measure space in the above sense. We select the identity element $\te$ as the basepoint for the polar coordinate system. Recall (see Theorem \ref{th:pCg-hpc}) that any such group $\G$ admits a horizontal polar coordinate system with respect to the quasi-norm $\boldN_u$ associated to Folland's fundamental solution $u$. Namely, there exists a $\mu$-null set $Z \subset \G$ so that $\G \setminus Z = \coprod_{\ta \in S \setminus Z} \gamma_\ta$
where $S = \{ \boldN_u = 1\}$ and, for each $\ta \in S \setminus Z$, $\gamma_\ta:(0,\infty)\to \G$ is a horizontal curve with $\gamma_\ta(1) = \ta$ and $\gamma_\ta(s) \to \te$ as $s \to 0$. Moreover,
\begin{equation}\label{eq:Ns}
\boldN_u(\gamma_\ta(s)) = s
\end{equation}
and
$$
||\dot\gamma_\ta(s)||_0 = \frac1{||\nabla_0 \boldN_u(\ta)||_0} =: \frac1{\lambda_\ta}
$$
for all $\ta \in S \setminus Z$. Thus the polar curve $\gamma_\ta$ has constant speed $\lambda_\ta^{-1}$. We reparametrize the polar curves by arc length. For each $\ta \in S \setminus Z$, define $\beta_\ta:(0,\infty)$ to be
$$
\beta_\ta(s) = \gamma_\ta(\lambda_\ta \, s).
$$
Then $||\dot\beta_\ta(s)||_0 = 1$ for all $s>0$. The polar coordinate integration formula now reads
\begin{equation}\label{eq:hpc2}
\int_\G f(\tg) \, d\tg = \int_S \int_0^\infty f(\beta_\ta(s)) \, s^{Q-1} \, ds \, d\tilde\sigma(\ta) 
\end{equation}
for all $f \in L^1(\G)$, where $d\tilde\sigma(\ta) = ||\nabla_0 \boldN_u(\ta)||_0^Q \, d\sigma(\ta)$. The expression in \eqref{eq:hpc2} now matches the form of the strong rectifiable polar coordinate formula \eqref{eq:rpc}, with data $(S,\tilde\sigma,\Gamma_\te^\infty(S))$. Note, however, that the formula in \eqref{eq:Ns} is less transparent in this presentation; with respect to the new parametrization for the polar curves we have
$$
\boldN_u(\beta_\ta(s)) = s \, ||\nabla_0 \boldN_u(\ta)||_0 \, .
$$

\begin{example}
In the Heisenberg group $\Heis^n$, the modified spherical measure $\tilde\sigma$ on $S = \{ \tx = (z,t) \in \Heis^n \, : \, \boldN(\tx) = 1\}$ is given by $d\tilde\sigma(z,t) = |z|^{2n+2} \, d\sigma(z,t)$, where $d\sigma$ is the measure appearing in \ref{eq:KR-polar-coordinates}.
\end{example}

The results of \cite{kn:polar} yield the following conclusions in the setting of polarizable Carnot groups. Recall that, if $(X,d,\mu)$ is $Q$-regular, then the $p$-modulus of $\Gamma^\infty$ is equal to zero for any $p\ge Q$, see \cite[Proposition 5.3.3]{hkst:book}.

\begin{corollary}[cf.\ Theorem 1.3 in \cite{kn:polar}]
Let $\G$ be a polarizable Carnot group with homogeneous dimension $Q$, and let $1\le p<Q$. Then
\begin{itemize}
\item[(1)] for any $u \in \dot{N}^{1,p}(\G)$ there exists a constant $c_u \in \R$ so that the limit $\lim_{s \to \infty} u(\beta_\ta(s))$ exists and equals $c_u$ for $\tilde\sigma$-a.e.\ $\ta \in S \setminus Z$, and
\item[(2)] for each $E \subset S \setminus Z$ with $\tilde\sigma(E)>0$, the $p$-modulus of $\Gamma_\te^\infty(E) := \{ \gamma_\ta \in \Gamma_\te^\infty(S) : \ta \in E\}$ is positive.
\end{itemize}
\end{corollary}

\bibliographystyle{acm}
\bibliography{biblio}
\end{document}